\documentclass[preprint]{elsarticle}

\RequirePackage{amsthm,amsmath,amsfonts,amssymb}
\usepackage[colorlinks=true, allcolors=blue]{hyperref}
\usepackage{booktabs,multirow,enumerate}
\usepackage{changepage} 

\newtheorem{lem}{Lemma}[section]

\newtheorem{thm}{Theorem}[section]
\newtheorem{example}{Example}[section]

\newtheorem{rem}{Remark}[section]
\numberwithin{equation}{section}
\numberwithin{figure}{section}
\numberwithin{table}{section}
\theoremstyle{plain}

\renewcommand{\d}{{\rm{d}}}

\newcommand{\beq}{\begin{equation}}
\newcommand{\eeq}{\end{equation}}

\newcommand{\var}{\text{\normalfont Var}}
\newcommand{\cov}{\text{\normalfont Cov}}
\newcommand{\ee}{\mathbb{E}}
\newcommand{\pp}{\mathbb{P}}
\newcommand{\rr}{\mathbb{R}}
\newcommand{\nn}{\mathbb{N}}
\newcommand{\zz}{\mathbb{Z}} 
\newcommand{\Qq}{\mathbb{Q}}

\def\qed{\rule{2mm}{2mm}}

\theoremstyle{plain}

\theoremstyle{remark}

\begin{document}

\begin{frontmatter}

\title{Least Squares-Based Permutation Tests in Time Series}

\author[stan1]{Joseph P. Romano\corref{cor1}}
\ead{romano@stanford.edu}

\author[stan2]{Marius A. Tirlea}
\ead{mtirlea@stanford.edu}

\cortext[cor1]{Corresponding author}

\address[stan1]{Departments of Statistics and Economics, Stanford University}

\address[stan2]{Department of Statistics, Stanford University}

\begin{abstract}
This paper studies permutation tests for regression parameters in a time series setting, where the time series is assumed stationary but may exhibit an arbitrary (but weak) dependence structure.
In such a setting, it is perhaps surprising that permutation tests can offer
any type of inference guarantees, since permuting of covariates can destroy their relationship with the response. Indeed, the fundamental assumption of exchangeability of errors required for the finite-sample exactness of permutation tests can easily fail. However, we show that permutation tests may be constructed which are asymptotically valid for a wide class of stationary processes, but remain exact when exchangeability holds.
We also consider the problem of testing for no monotone trend and we construct asymptotically valid permutation tests in this setting as well.
In addition, in order to use the methods, the
\texttt{R} package \texttt{permixOLS} is publicly available.
\end{abstract}


\end{frontmatter}

\section{Introduction}

In every quantitative field, an issue of critical importance is the assessment of the relationship between explanatory variables and a response variable. The classical method most frequently used to undertake such an analysis is that of ordinary least squares (OLS) regression, which involves construction of a feature matrix $X$ and a response vector $Y$, and computation of the sample regression coefficient $\hat{\beta}$ as an estimate of the unconditional linear regression coefficient $\beta$ between the covariates $X$ and the response variable $Y$. The computation of $\hat{\beta}$ is often accompanied by a test of the null hypothesis 

\beq \label{equation:null1}
H_0: \beta = 0 \,\, ,
\eeq

\noindent since it is often of great interest to determine which of the regression coefficients differ significantly from 0. In general, such tests compare the sum of the squared residuals to a reference null distribution, and in doing so make parametric distributional assumptions both on the pairs of samples $(X_i, \, Y_i)$, often assuming that the $(X_i,\, Y_i)$ are jointly Gaussian, and also often assume that the pairs of samples are independent. In the context of time series analysis, neither of these assumptions hold in general, and so any accompanying inference from such a hypothesis test may be wildly incorrect. 

In addition, OLS regression is often used to test for monotone trend in time series analysis: namely, for a sequence $\{X_i: i \in [n]\}$, a regression of the sequence $\{X_i: i \in [n]\}$ against index $\{i : i \in [n]\}$ is performed, and a univariate $t$-test is performed to assess whether or not the resulting regression coefficient is significantly different from 0. However, there is an implicit assumption under the null hypothesis that the $X_i$ are i.i.d., and so such a test can lead to rejection of the null hypothesis of ``no monotone trend", for some definition thereof, on account of the violation of independence, as opposed to because there is truly some monotone trend.

In this paper, we propose nonparametric testing methods for both the classical regression null hypothesis of no significance  
(\ref{equation:null1})
and for the linear regression null hypothesis of no monotone trend 

$$
H_0 ^{(m)}: \{X_i: i \in [n]\} \text{ exhibits no monotone trend,} 
$$

\noindent where the definition of a lack of monotone trend will be stated precisely.

The problems described above have been studied extensively in different contexts. \cite{durbin} estimates regression parameters in a time series setting, \cite{tsay} considers a linear model in which the errors are distributed according to an autoregressive process, \cite{tran.et.al} consider recovery of a regression function in a regression setting with time-dependent errors, \cite{kedem} discusses several different methods for regression in a time series context, and \cite{wavk} and \cite{wavk.wrong.one} consider a nonparametric test in the context of time series regression analysis.

In the setting of permutation-based tests for linear regression, \cite{anderson.robinson} produce a permutation test based on partial correlations, \cite{nyblom} discusses several regression-based permutation tests, \cite{diciccio2017} discuss testing for correlation and regression coefficients, and, in a more empirical setting, \cite{hemerik} performs an empirical analysis of permutation testing in a high-dimensional linear regression problem.  In \cite{tirlea},  permutation tests are developed for testing autocorrelations in a time series setting.

None of these papers consider the least squares regression problem when the observations may be time-dependent.  Indeed, permutation tests in linear models have only been considered when the errors are i.i.d.  In contrast, we allow for quite general time-dependent stationary processes.
Needless to say, the class of i.i.d. processes, even if the distribution of errors is not modeled parametrically, is  a dramatically smaller class than the stationary processes considered here.
Thus, the technical arguments are more challenging than the i.i.d. case considered in \cite{diciccio2017}.   It is perhaps even surprising that permutation tests, which totally destroy the time-ordering of observations, can be used in such a general context.  Other resampling methods, like the block bootstrap or stationary bootstrap, require generating new data sets that implicitly capture the underlying dependence structure of the observations.   Yet,  we will obtain asymptotic validity under weak conditions, as well as obtain an exact finite-sample method that holds when errors are i.i.d. 

Thus,
the goal of this paper is to develop a framework for the use of permutation testing as a valid method for testing the two aforementioned hypotheses, which preserves the exactness property under the randomization hypothesis, but is also asymptotically valid for a large class of weakly dependent sequences.
Permutation tests (as a special case of randomization tests) provide exact level $\alpha$ tests when the randomization hypothesis holds; see \cite{tsh}, Chapter 17.  As such, they offer a means of
valid inference in some complex semi-parametric and nonparametric settings.  However,  the problem is they often fail if the randomization hypothesis does not hold.  For example, when testing a null hypothesis that  a regression coefficient  is 0,  the randomization hypothesis holds for permutations of the regressors when the regressors are independent of the response.    But of course, the regression coefficient may be 0 even under dependence, in which case Type 1 error rates can be severely inflated, even asymptotically.
But a growing literature as described in Section 1 of \cite{tirlea} shows that asymptotic validity can be recovered for certain (studentized) choice of test statistics,  while preserving exactness in finite samples when the randomization hypothesis holds.
The challenge in the context of this paper is that we are allowing the regressors and response to be dependent, but we are also allowing observations across time to be dependent.  In such settings, it is not clear that permutations should offer a valid approach to inference, but our findings are indeed positive.

In order to provide some intuition for the results,   consider the case of a test statistic that has some limiting distribution under the null hypothesis, say $F_0$
(which may of course depend on unknown parameters).   The permutation distribution, formally defined later in (\ref{equation:Rnt}), is the empirical distribution of  the test statistic  recomputed over all permutations of  the covariates $X$.  Since the permutation test is exact in the i.i.d.  case when  $X$ and the response $Y$ are independent,  it is reasonable to expect that the permutation distribution will approximate $F_0$ in this case, at least asymptotically.   But we want to consider settings where $X$ and $Y$ are dependent, and furthermore when serial dependence may be present.  In such case, the permutation distribution destroys any relationship between $X$ and $Y$, and so the permutation distribution would approximately behave like it would based on i.i.d.  data $( \tilde X , \tilde Y  )$, where $\tilde X$ has the same distribution as $X$, $\tilde Y$ has the same distribution as $Y$, but $\tilde X$ and $\tilde Y$ are independent.    In general, the limiting  permutation distribution may differ from the true limiting distribution, since the limiting distributions under $(X,Y)$  and $( \tilde X , \tilde Y )$ will in general differ.    However, if no matter what the dependence, a test statistic is constructed that has a limiting distribution that does not depend on any unknown parameters (which is sometimes called asymptotically pivotal), then the permutation distribution will approximate this common distribution regardless of the dependence (under weak assumptions).   One way to construct  an asymptotically pivotal test statistic is to consider an estimator that is asymptotically normal and studentize it by dividing by an appropriate estimator of the standard deviation, as long as it is a valid estimate under general dependence conditions.  This will be our approach and the technical results support this intuition.

Section \ref{sec.prelim.ols} provides several useful preliminary definitions and results. The main results relating to the null hypothesis $H_0$ are given in Section \ref{sec.ols.main}, in which we give conditions for the asymptotic validity of the permutation test when $\{X_i: i \in [n]\}$ is a stationary, $\alpha$-mixing sequence, satisfying relatively standard technical moment and mixing conditions.   Applications include inference for regression parameters, autocorrelations and cross-correlations. Section \ref{sec.ols.monotone} gives the main results for the asymptotic validity of the permutation test for the hypothesis of no monotone trend. 
The use of a regression estimator for detecting trend has been commonly suggested; see \cite{good} and \cite{debunk}.
However, the problem and certain technical aspects of the problem here are quite distinct from Section \ref{sec.ols.main}, and so these results are not a special case of previous results.
In particular,  the index regressor is not stationary and unbounded.  More important, the class of alternatives is now nonstationary, which can cause problems for permutation tests; see Example \ref{example:41}.  However, we do obtain some positive results by imposing a common restriction on the class of alternatives, though the class is still quite large.
Section \ref{sec.ols.sim} provides simulations illustrating the results. On account of the technical nature of showing the results in the aforementioned sections, the proofs are deferred to the supplement.
The methods of Sections \ref{sec.ols.main} and \ref{sec.ols.monotone} are implemented in the \texttt{R} package \texttt{permixOLS}, available at \url{https://github.com/mtirlea/permixOLS}, which includes worked examples
reproducing every quantity of the regression test on a small dataset.

\section{Preliminaries}\label{sec.prelim.ols}

In this section, we establish notation and definitions which will be used in Section \ref{sec.ols.main}, in addition to providing a useful central limit theorem. We begin with a brief discussion of $\alpha$-mixing. Let $\{X_n, \, n \in \zz \}$ be a stationary sequence of random variables, adapted to the filtration $\{\mathcal{F}_n\}$. Let $\mathcal{G}_n = \sigma\left(X_r : \, r \geq n\right)$. For $n \in \nn$, let $\alpha_X(n)$ be Rosenblatt's $\alpha$-mixing coefficient, defined as 

\beq
\alpha_X (n) = \sup_{A \in \mathcal{F}_0, \, B \in \mathcal{G}_n} \left| \pp(A \cap B) - \pp(A) \pp(B)  \right | \, \, .
\eeq

\noindent We say that $\{X_n\}$ is $\alpha$-mixing if $\alpha_X (n) \to 0$ as $n \to \infty$. 

For an $(\Omega, \, \mathcal{F}, \, \pp)$-measurable $\rr^d$-valued random variable $X$, and $p \in \rr^+$, let $\lVert X \rVert_p$ denote the $L^p$-norm of $X$, i.e. 

$$
\lVert X \rVert_p  =  \ee \left[ \sum_{i=1}^d |X_i| ^p \right] ^{1/p}\, \, .
$$ 

This paper is concerned about inference for $\beta$ in least squares regression.  That is,
given observations $\{(Y_i, \, X_i): i \in [n]\} \subset \rr \times \rr^p$, let $\hat \beta_n$ be the usual
least squares coefficient  of the regression of $Y$ on $X$ defined by 

\beq  \label{equation:betanhat}
\hat{\beta}_n = \left[ \frac{1}{n} \sum_{i=1}^n (X_i - \bar{X}_n )(X_i - \bar{X}_n )^T \right] ^{-1} \left( \frac{1}{n} \sum_{i=1} ^n (Y_i - \bar{Y}_n) (X_i - \bar{X}_n) \right)~~,
\eeq
 where $\bar{X}_n$ is the coordinate-wise sample mean of the $X_i$, and  $\bar{Y}_n$ is the sample mean of the $Y_i$.  Also let   $\beta$ be the unconditional regression coefficient of $Y_1$ on $X_1$, i.e. 

\beq
\beta = \Sigma_X ^{-1} \cov(Y_1, X_1) \, \, ,
\label{beta.ols.defn}
\eeq
where $\Sigma_X$ is the covariance matrix of $X_i$, assumed positive definite.  Also,  define
the residuals
\beq
\hat{\epsilon}_{i, \,n} = Y_i - X_i^T \hat{\beta}_n \, \, .
\label{residual}
\eeq

 Exact permutation tests (and randomization tests more generally) can be constructed when    the randomization hypothesis holds (\cite{tsh}, Chapter 17).  In our setting, the randomization hypothesis refers to the assertion that transformations of the data by certain permutations leaves the underlying distribution unchanged, if the null hypothesis holds.   Therefore, 
  it is  important to define the action of a permutation on pairs $\{(Y_i, \, X_i): i \in [n]\} \subset \rr \times \rr^p$.  To be sure, the randomization hypothesis will not generally hold in our setting, and the main problem becomes how to adapt a permutation test so that it is at least asymptotically valid while retaining the exactness property when the randomization hypothesis holds.

 Let $S_n$ be the symmetric group of all  permutations of $[n] = \{ 1, \ldots , n \}$.
 We say that $\pi_n \in S_n$ acts on $\{Z_i = (Y_i, \, X_i): i \in [n]\}$ by row-wise permutation of the $X_i$, i.e. 

$$
Z_{\pi_n} = ( (Y_1, \, X_{\pi_n(1), \, 1}, \, \dots,\, X_{\pi_n(1), \, p}), \, \dots,\, (Y_n, \, X_{\pi_n(n), \, 1}, \, \dots,\, X_{\pi_n(n),\, p}) ) \, \, .
$$

\noindent In the context of linear regression, the interpretation of the action of this permutation is that we permute the rows of the feature matrix $X$ prior to computation of the regression coefficient. Note also that, in the case of the action of a random permutation $\Pi_n \sim \text{Unif}(S_n)$ on $Z_n$, this definition of the action of a permutation is equivalent to the action of two i.i.d. random permutations on the $Y_i$ and $X_j$, since, for $\Pi_n, \, \Pi_n' \sim \text{Unif}(S_n)$, with $\Pi_n$ and $\Pi_n'$ independent

$$
(\Pi_n')^{-1} \Pi_n \overset{d}= \Pi_n \, \, .
$$

\noindent By a similar argument to Lemma S.3.1 in \cite{tirlea}, in the setting in which the sequence $\{Z_i: i \in [n]\}$ is $\alpha$-mixing, the randomization hypothesis holds if and only if the sequences $\{Y_n: n \in \nn\}$ and $\{X_n: n\in \nn\}$ are independent and the $X_i$ are exchangeable, which, in the setting of stationary and $\alpha$-mixing sequences, is equivalent to the $X_i$ being i.i.d.   If such independence holds,  the usual OLS regression parameter vector $\beta$  is 0 (assuming that the second moment of $X$ is finite so that $\beta$ is  well-defined),  and permutation tests are exact.  However,  $\beta$ can be 0 without the randomization hypothesis holding,  and we still wish to construct valid procedures in such settings.

Corresponding to some test statistic $T_n$ (possibly multivariate, such as $\sqrt{n} \hat \beta_n$), let the permutation distribution $\hat R_n ( \cdot )$ be defined, for any vector $t$, by
\beq
\hat R_n ( t)  = \frac{1}{n!} \sum_{\pi_n \in S_n } I \{ T_n \le t \}~.
\label{equation:Rnt}
\eeq
When $T_n$ is a real-valued test statistic, the corresponding permutation test rejects when $T_n$ exceeds the $1-  \alpha$ quantile of $\hat R_n ( \cdot )$. Under the randomization hypothesis, this test controls the Type 1 error at level $\alpha$, though it may be slightly conservative due to ties and the fact that $n! \alpha$ may not be an integer.  In such case, the test can be modified to achieve exact level $\alpha$  by randomizing in the usual way; see Section 17.2 in \cite{tsh}.

In Section \ref{sec.ols.main},  we will compare the limiting sampling distribution of the test statistic to the permutation distribution (\ref{equation:Rnt}) in order to determine asymptotic validity of the permutation test.  After defining notation, we first  provide the limiting distribution of $\hat \beta_n$, suitably normalized.   The proof is in the supplement and makes use of the central limit theorem due to \cite{neumann}.  A similar result to Theorem \ref{ols.test.stat.clt} has been shown in Chapter 5 of \cite{white.ols}, with slightly different conditions on the sequence $\{Z_n = (Y_n, \, X_n): n \in \nn\}$. It is provided in this section to make explicit the limiting distribution of the sample least squares coefficient $\hat{\beta}_n$, as well as establish the notation and conditions for comparing the true sampling distribution with the permutation analogue.

\medskip

\noindent{\it Notation:}  Let $\Gamma$ be 
 the $p \times p$ covariance matrix defined by

\beq
\Gamma = \var(\tilde X_1 ( \tilde Y_1 - \tilde X_1 ^T \beta) ) + 2\sum_{n \geq 2} \cov(\tilde X_1 ( \tilde Y_1 - \tilde X_1 ^T \beta), \, \tilde X_n ( \tilde Y_n - \tilde X_n ^T \beta) )~,
\label{gam.ols.true.var}
\eeq
where $\tilde X_i = X_i - \ee (X_i )$ and $\tilde Y_i = Y_i - \ee (Y_i)$.
Let   $\Sigma_X$ be the covariance matrix of $( X_{n,1} , \ldots , X_{n,p} )^T$.
Also, let $\hat \Sigma_X$ be the sample covariance matrix of the $X_i$, computed with divisor $n$ as in (\ref{equation:betanhat}).

\begin{thm} 
For each $n$, let $Z_n = (Y_n, \, X_{n, \, 1}, \, \dots, \, X_{n, \, p})$.  Assume  $\{Z_n: n \in \nn\}$ is a stationary sequence of $\rr^{p+1}$-valued random variables. 
Suppose that, for some $\delta > 0$, 

$$
\lVert Z_1 \rVert_{8 + 4\delta} < \infty
$$

\noindent and 

$$
\sum_{n \geq 1} \alpha_Z (n) ^\frac{\delta}{2 + \delta} < \infty \, \, .
$$
\noindent  Assume $\Gamma$ defined in (\ref{gam.ols.true.var}) 
 is positive definite, as is $\Sigma_X$.  Then, with $\hat \beta_n$ defined in (\ref{equation:betanhat}),  as $n \to \infty$,

$$
\sqrt{n} (\hat{\beta}_n - \beta) \overset{d}\to N(0, \, \Sigma_X ^{-1} \Gamma \Sigma_X^{-1} ) \, \, .
$$

\label{ols.test.stat.clt}
\end{thm}

Having obtained the above result, we next will  compare the limiting sampling distribution and the limiting behavior of the permutation distribution in order to determine whether or not asymptotic validity holds.

\section{Testing for regression coefficients}\label{sec.ols.main} 

In this section, we consider the problem of testing the null hypothesis 
(\ref{equation:null1}),
where $\beta = \beta(Y, \, X)$ is the unconditional regression coefficient defined in (\ref{beta.ols.defn}), in the setting where the sequence $\{(Y_i,\, X_i): i \in \nn\} \subset \rr\times \rr^p$ is stationary and $\alpha$-mixing. As discussed in Section \ref{sec.prelim.ols}, when the randomization hypothesis holds, one may construct permutation tests of the hypothesis $H_0$ with exact level $\alpha$ in finite samples.

However, if the randomization hypothesis does not hold, the test may not be valid even asymptotically, i.e. the rejection probability may not be equal to or even close to $\alpha$ in the limit as $n\to \infty$. Our goal is therefore to construct a permutation testing procedure, based on an appropriately chosen test statistic, which has asymptotic rejection probability equal to $\alpha$, and which retains the finite sample exactness property under the additional assumption that the randomization hypothesis holds. 

In a similar fashion to \cite{tirlea}, we wish to consider a permutation test based on the sample regression coefficient $\hat{\beta}_n$. As in that paper, our overarching goal is show that, for an appropriately studentized version of $\hat{\beta}_n$, under the null hypothesis $H_0: \beta = 0$, the asymptotic distribution of the test statistic is asymptotically pivotal, i.e. does not depend on the underlying process $\{(Y_n,\, X_n)\}$, and that the corresponding permutation distribution $\hat{R}_n$ has a limiting distribution which is the same as the limiting distribution of the test statistic.

We begin with analysis of the unstudentized permutation distribution corresponding to the statistic $\sqrt{n} \hat \beta_n$.

\begin{thm} Under the assumptions of Theorem \ref{ols.test.stat.clt}, let $\hat{R}_n$ be the permutation distribution defined in (\ref{equation:Rnt}) based on the test statistic $\sqrt{n} \hat{\beta}_n$. Then, as $n \to \infty$, 

$$
\sup_{t \in \rr^p} \left| \hat{R}_n (t) - F_{\sigma_Y^2 \Sigma_X ^{-1} } (t) \right| \overset{p} \to 0 \, \, .
$$

\noindent where, for a positive definite matrix $M\in\rr^{p\times p}$, $F_{M}$ is the multivariate c.d.f. of a $N(0, \, M)$ random variable.

\label{ols.perm.unstud.thm}
\end{thm}

The proof and some technical remarks are given in the supplement.   Observe that, while the limiting distribution of the sample regression coefficient and the limiting permutation distribution are both multivariate Gaussian with mean zero under the null hypothesis $H_0$, the covariance matrices of these two distributions are not equal in general.
Hence, the permutation test may {\it not} be asymptotically valid. We therefore proceed to construct an appropriate studentizing factor so that, ultimately, the asymptotic distribution of the studentized statistic
as well as its permutation counterpart converge to the same limit, ensuring asymptotic validity.  First,  we require the following result.

\begin{lem}  Assume  the setting of Theorem \ref{ols.test.stat.clt}.
Let $\{b_n: n \in \nn\} \subset \nn$ be a nondecreasing sequence such that $b_n \to \infty$ as $n \to \infty$ and $b_n = o(\sqrt{n})$. For each $n$, for each $(i, \, j) \in [n] \times [n]$, let 

$$
\begin{aligned}
\hat{\gamma}_{n, \, i, \, j} &= \frac{1}{n} \sum_{l = 1}^n \left(X_{l, \, i} \hat{\epsilon}_{l, \, n} - \frac{1}{n} \sum_{k=1}^n X_{k, \, i} \hat{\epsilon}_{k, \, n} \right)\left(X_{l, \, j} \hat{\epsilon}_{l, \, n} - \frac{1}{n} \sum_{k=1}^n X_{k, \, j} \hat{\epsilon}_{k, \, n} \right) \\
\hat{\tau}_{n, \, i, \, j} &= \frac{2}{n} \sum_{k = 1}^{b_n} \sum_{l = 1}^{\min\{n-k, \, n- b_n\}}  \left(X_{l, \, i} \hat{\epsilon}_{l, \, n} - \frac{1}{n} \sum_{r=1}^n X_{r, \, i} \hat{\epsilon}_{r, \, n} \right)\left(X_{l+k, \, j} \hat{\epsilon}_{l +k, \, n} - \frac{1}{n} \sum_{r=1}^n X_{r, \, j} \hat{\epsilon}_{r, \, n} \right)
\end{aligned}
$$
where the residuals $\hat \epsilon_{i,n}$ are defined in (\ref{residual}).  Let $\hat{\Gamma}_n \in \rr^{p \times p}$ be such that 

\begin{equation}\label{equation:hatGamma}
\hat{\Gamma}_{i, \, j} = \hat{\gamma}_{n,\, i, \, j} + \hat{\tau}_{n, \, i, \, j} \, \, .
\end{equation}

\noindent Let $\Pi_n \sim \text{Unif}(S_n)$, independent of $\{Z_i: i \in[n]\}$, and let 

$$
\hat{\Gamma}_{\Pi_n} = \hat{\Gamma}_n((Y_1, \, X_{\Pi_n (1)}), \, \dots, \, (Y_n, \, X_{\Pi_n( n) }))
$$

\noindent Then, as $n \to \infty$, we have that 

$$
\begin{aligned}
\hat{\Gamma}_n &\overset{p}\to \Gamma \\
\hat{\Gamma}_{\Pi_n} &\overset{p}\to \sigma_Y ^2 \Sigma_X \, \, .
\end{aligned}
$$

\label{lem.ols.studentize}
\end{lem}

\begin{rem} \rm 
Evidently, the estimator $\hat \Gamma_n$ is a consistent estimator of $\Gamma$.  On the other hand, $\hat \Gamma_{\Pi_n}$ is the estimator of $\Gamma$, but computed on a randomly permuted data set.    When $\Pi_n$ is the identity permutation, the estimators agree.  But the conclusion of the lemma shows their large sample behaviors differ in general.

We observe that the estimator $\hat{\Gamma}_n$ is quite similar to the Newey-West estimator (\cite{newey.west}), and similar in spirit to the heteroskedasticity-consistent estimator of \cite{white.est}. Indeed, an analogous proof to the one of Lemma \ref{lem.ols.studentize} shows that the Newey-West estimator of variance satisfies the same limiting properties as the chosen estimator $\hat{\Gamma}_n$. \qed
\end{rem} 

The following result shows that a properly studentized test statistic leads to asymptotic validity of permutation tests, since now the limiting covariance matrices match.    The proof follows by combining Theorem \ref{ols.test.stat.clt}, Theorem \ref{ols.perm.unstud.thm},  Lemma \ref{lem.ols.studentize}, Slutsky's theorem and  the multivariate Slutsky theorem for randomization distributions (\cite{chung.multivariate}, Lemma A.5).

\begin{thm} 
Under the assumptions of Theorem \ref{ols.test.stat.clt} and Lemma \ref{lem.ols.studentize}, the following hold.

\noindent (i)  As $n\to\infty$,

$$
\left(\hat{\Gamma}_n^{-1/2} \hat{\Sigma}_X \right)\sqrt{n} (\hat{\beta}_n - \beta) \overset{d} \to N(0, \, I) \, \, .
$$

\noindent (ii)  Let $\hat{R}_n$ be the permutation distribution based on the test statistic $\left(\hat{\Gamma}_n^{-1/2} \hat{\Sigma}_X \right)\sqrt{n} \hat{\beta}_n$, with associated group of transformations $S_n$, the symmetric group of order $n$. Then, as $n\to\infty$, 

$$
\sup_{t \in \rr^p} \left| \hat{R}_n(t) - F_I (t) \right| \overset{p} \to 0 \, \, ,
$$

\noindent where $F_I$ is the multivariate c.d.f. of the $p$-dimensional Gaussian distribution with zero mean and identity variance matrix.

\label{ols.perm.big.thm}
\end{thm}

\begin{rem} \rm
Note that $\hat{\Gamma}_n$, as defined in (\ref{equation:hatGamma}), need not be
symmetric for $p > 1$. Accordingly, $\hat{\Gamma}_n^{-1/2}$ above is to be read as shorthand
for the truncated inversion described in Section \ref{sec.ols.sim}. \qed
\label{rem.gamma.asym}
\end{rem}

 As a result of Theorem \ref{ols.perm.big.thm} and the Continuous Mapping Theorem, we may construct an asymptotically valid permutation test of the null hypothesis $H_0$ at the nominal level $\alpha$ as follows. 
Let $Z_n = \left(\hat{\Gamma}_n^{-1/2} \hat{\Sigma}_X \right)\sqrt{n} \hat{\beta}_n$.
Also, let $\| \cdot \|$ be any norm on $\rr^p$, such as the usual Euclidean norm.  Then, we can apply a permutation test
using the test statistic $\| Z_n  \| $.
In fact, any continuous function $g$ of the components of $Z_n$
can be used as a test statistic.

\begin{rem} \rm 
Of course, we may also use part (i) of Theorem \ref{ols.perm.big.thm}  to get an asymptotically valid test as follows.  To that end, let $P$ be the probability measure induced by the standard multivariate Gaussian distribution on $\rr^p$. 
Let $A$ be any Borel set in $\rr^p$ whose boundary has Lebesgue measure zero.  If  $P(A) = 1- \alpha$, then the test that rejects if 

$$
\left(\hat{\Gamma}_n^{-1/2} \hat{\Sigma}_X \right)\sqrt{n} \hat{\beta}_n \in A^C \, \, ,
$$

\noindent where $A^C = \{x \in \rr^p: x \notin A\}$, and does not reject otherwise is asymptotically valid at the nominal level $\alpha$. For instance, letting $A$  be the ball of an appropriate radius centered at the origin will correspond to the Chi-squared test of significance in standard OLS regression theory. Similarly, taking $A$ to be a cube centered at the origin (or   more generally a hyper-rectangle)  corresponds, in a traditional regression setting, to the test which rejects for large values of the maximum. Alternatively, one may take $A$ to be an ellipsoid centered at the origin, if one wishes to place more emphasis on a particular coordinate of the unconditional regression coefficient than the others.
In fact,  any measurable set $A$ satisfying the properties described above will induce an asymptotically valid test. To provide another, more general example, for a continuous function $g: \rr^p \to \rr$, and a choice of $t \in \rr$ such that, for $Z \sim N_p(0, \, I_p)$,

$$
\pp(g(Z)\leq t) = 1- \alpha \, \, ,
$$

\noindent we may take $A = g^{-1} ((-\infty, \, t])$.
We emphasize, however, that this latter asymptotic approach does not share the exactness properties of permutation tests when the randomization hypothesis holds.  \qed

\label{rem.rej.ols}
\end{rem}

\begin{example} ($VARMA(s, \, q)$ process) \rm Let $d \in \nn$, and let $\{\epsilon_n: n \in \nn\} \subset \rr^{d+1}$ be a sequence of mean zero i.i.d. random variables such that, for some $\delta > 0$, 

$$
\lVert \epsilon_1 \rVert_{8 + 4\delta} < \infty \, \, .
$$

\noindent Let $A_0 = B_0 = I_{d+1}$ be equal to the identity matrix. Let $s, \,q \in \nn$, and let $\{A_i: i \in [q]\} \subset \rr^{(d + 1) \times (d+1)}$ and $\{B_j: j \in [s]\} \subset \rr^{(d+1)\times (d+1)}$ be matrices.

Let $\{Z_n = (Y_n, \, X_{n,\, 1}, \, \dots,\, X_{n,\, d}): n \in \nn\} \subset \rr^{d+1}$ satisfy the following vector autoregressive moving average $(VARMA)$ equation, i.e. for all $n \in \nn$ such that $n \geq s +1$, 

\beq
\sum_{j = 0}^s B_j Z_{n - j} = \sum_{i=0} ^qA_i \epsilon_{n- i} \,\, ,
\label{varma.eq}
\eeq

\noindent i.e. $\{Z_n\}$ is a $VARMA(s, \, q)$ process. Let $g: \mathbb{C} \to \mathbb{C}$ be the polynomial 

$$
g(z) = \det \left( \sum_{j = 0}^s B_j z^j \right) \, \, .
$$

\noindent Suppose that the roots of $g$ lie outside the complex unit circle $\mathbb{D} = \{z \in \mathbb{C}: |z| \leq 1\}$. Then, by \cite{lutkepohl}, there exists a unique stationary solution to (\ref{varma.eq}), i.e. there exists a stationary sequence $\{Z_n\}$ satisfying (\ref{varma.eq}). 

Suppose that, in this setting, the unconditional regression coefficient $\beta = 0$. As long as the $\epsilon_i$ have an absolutely continuous density with respect to the Lebesgue measure on $\rr^{d+1}$, it follows that, by Theorem 1 of \cite{mokkadem}, the sequence $\{Z_n\}$ is geometrically $\beta$-mixing, and so the conditions of Theorem \ref{ols.perm.big.thm} are satisfied. In particular, for any $\alpha \in (0,\,1)$, for any choice of null region $A$ as described in Remark \ref{rem.rej.ols}, the permutation test of the null hypothesis $H_0$ with null region $A$ will be asymptotically valid at the nominal level $\alpha$. \qed

\label{varma.ols}
\end{example}

\begin{rem} \rm Note that Theorem \ref{ols.perm.big.thm} also permits for the inclusion of lagged regressor variables; indeed, consider $\{(Y_i, \, X_i): i \in [n]\}$ in the setting of Theorem \ref{ols.perm.big.thm}. We may, for instance, introduce a new column to the data matrix $X$, of the form

$$
X_{i, \, p+1} = X_{i-1, \, 1} \, \, .
$$ 

\noindent The mixing and moment conditions of Theorem \ref{ols.perm.big.thm} still hold for the updated versions of the $X_i$, and so we may perform analogous inference with the inclusion of this new regressor. \qed

\label{rem.lags.ols}
\end{rem}

\begin{example} (Ljung-Box permutation test) \rm Let $\{T_n: n \in \nn\}$ be a real-valued, stationary, $\alpha$-mixing sequence satisfying the moment and mixing conditions of Theorem \ref{ols.perm.big.thm}. Let $p \in \nn$, and, for each $n \in \nn$ and $j \in [p]$, let 

$$
\begin{aligned}
Y_n &= \frac{T_{n+p}}{\hat{\sigma}_T}\\
X_{n, \, j} &= \frac{T_{n + p - j}}{\hat{\sigma}_T}  \, \, ,
\end{aligned}
$$

\noindent where $\hat{\sigma}_T$ is the sample standard deviation of $\{T_i: i \in [n +p]\}$. The sequence $\{Z_n = \hat{\sigma}_T(Y_n,\, X_n): n \in \nn\}$ satisfies the mixing and moment conditions of Theorem \ref{ols.perm.big.thm}, and $\hat{\sigma}_T$ is invariant under permutations. Hence we may apply Remark \ref{rem.lags.ols}, to construct a permutation test analogous to the Ljung-Box test. To be precise, consider a hypothesis test of the hypothesis 

$$
H: \rho_1 = \dots = \rho_p = 0 \, \, ,
$$

\noindent where $\rho_j$ is the $j$th order autocorrelation of the sequence $\{T_n: n \in \nn\}$. Let 

$$
\hat{\rho}_n = (\hat{\rho}_{n,\, 1}, \, \dots,\, \hat{\rho}_{n,\, p}) \, \, ,
$$

\noindent where $\hat{\rho}_{n,\, j}$ is the $j$th sample autocorrelation of the sequence $\{T_i: i \in [n + p]\}$. Since the sample variance of the $T_i$ is invariant under permutations, by Theorem \ref{ols.perm.big.thm}, the permutation test based on the test statistic 

$$
\sqrt{n} \hat{\Gamma}_n^{-1/2}\hat{\rho}_n \, \, ,
$$

\noindent which rejects for large values of the statistic 

$$
n \hat{\rho}_n ^T \hat{\Gamma}_n^{-1} \hat{\rho}_n \, \, ,
$$

\noindent is an asymptotically valid  test of $H$ at the nominal level $\alpha$. While this is not exactly of the same form as the Ljung-Box test, this is similar, in the sense that the test statistic is a quadratic form of the sample autocorrelations. \qed

\label{ex.ljung.box.ols}
\end{example}

\begin{example} (Testing for cross-correlations) \rm Let $\{({Y}_n,\, {X}_n): n \in \nn\}$ be as in Theorem \ref{ols.perm.big.thm}. Consider a test of the null hypothesis 

$$
H: \rho_{Y,\, {X}_1} = \dots =\rho_{{Y},\, {X}_p} = 0\, \, ,
$$

\noindent where, for each $j \in [p]$, $\rho({Y},\, {X}_j) = \text{corr}({Y}_1,\, {X}_{1,\,j})$.

As in Example \ref{ex.ljung.box.ols}, since, for each $j\in [p]$, the coordinate-wise $j$th sample variance 

$$
\hat{\sigma}_{n,\, j} ^2 = \frac{1}{n} \sum_{i=1}^n ({X}_{i, \, j} - \bar{{X}}_{n,\, j}) ^2 
$$

\noindent is permutation invariant, we may construct a test  of $H$ based on permutations. 

 Also, let

$$
\hat{\rho}_n = (\hat{\rho}_{n,\, 1}, \, \dots,\, \hat{\rho}_{n,\,p} ) \, \, ,
$$

\noindent where, for each $j \in [p]$, $\hat{\rho}_{n,\, j}$ is the sample correlation between $\{Y_i: i \in [n]\}$ and $\{X_{i, \, j}: i \in [n]\}$. We have that the permutation test based on the test statistic 

$$
n \hat{\rho}_n^T \hat{\Gamma}_n ^{-1} \hat{\rho}_n \, \, ,
$$

\noindent will result in an asymptotically valid two-sided test of $H$ at the nominal level $\alpha$. In addition, by Remark \ref{rem.lags.ols}, such a test may also be used to test whether lagged cross-correlations of the form 

$$
\rho_{{Y}, {X}_j,\, r} = \text{corr}({Y}_1, \, {X}_{1 + r,\, j}) 
$$

\noindent are significantly different from zero. \qed

\end{example}

\section{Testing for monotone trend using least squares}\label{sec.ols.monotone}

In this section, we consider the nonparametric regression model 

\beq \label{equation:yni}
Y_{n,\, i} = X_{n,\, i} + \lambda_{n,\, i} \, \, ,
\eeq

\noindent where $\{X_{n,\, i}: n \in \nn\}$ is a row-wise stationary triangular array of $\alpha$-mixing real-valued random variables, and $\{\lambda_{n,\, i}: n \in \nn,\, i \in [n]\}$ is a triangular array of real-valued constants such that 

\beq \label{equation:lambdani}
\lambda_{n,\, i} = g\left( \frac{i}{n} \right) 
\eeq

\noindent for some (unknown) continuous monotone function $g$. Only the $Y_{n,i}$ are observed.  We focus on testing the null hypothesis that the triangular array $\{Y_{n,\, i}\}$ does not exhibit any monotone trend, i.e. 

$$
H_0 ^{(m)}: g(x) \equiv C \text{ for all } x \in [0, \, 1] \, \, .
$$

\noindent using a permutation testing procedure based on a possibly studentized version of the sample regression coefficient obtained when regressing the sequence $\{Y_i: i \in [n]\}$ on the index $\{i: i \in [n]\}$, i.e. 

$$
\hat{\beta}_n = \frac{1}{\sum_{i=1}^n \left(i - \frac{n+1}{2} \right)^2 } \sum_{i=1}^n \left( i - \frac{n+1}{2} \right) (Y_i - \bar{Y}_n) \, \, .
$$

Under the additional assumption that the triangular array $\{X_{n,\, i} \}$ is row-wise i.i.d., the randomization hypothesis holds, and so the permutation test based on any test statistic will be finite-sample exact at the nominal level $\alpha$. However, in the more general setting of $\alpha$-mixing sequences, the test may not be exact in finite samples, or indeed asymptotically valid as $n \to \infty$. This leads to issues of a lack of Type 1 and Type 3 error control, where the permutation test may reject even when the $\lambda_{n, \, i}$ are constant, on account of the triangular array of errors $X_{n,\,i}$ not being independent and identically distributed. 

Before we may commence closer examination of these issues, we provide some preliminary definitions and results, in addition to a brief discussion of the choice of problem setting.

\subsection{Preliminaries} 

In this section, we establish some preliminary notation, in addition to providing some simple preliminary results and a central limit theorem for the limiting distribution of the test statistic. 

We begin with some preliminary notation, to be used throughout the remainder of this section. In the problem setting (\ref{equation:yni}), 
 let  
$$
\begin{aligned}
\bar{X}_n &= \frac{1}{n} \sum_{i=1}^n X_{n,\,i} \\
\bar{\lambda}_n &= \frac{1}{n} \sum_{i=1} ^n\lambda_{n,\,i} \\
\bar{Y}_n &= \frac{1}{n} \sum_{i=1}^n Y_{n,\, i}
\end{aligned}
$$

We now provide a simple lemma to motivate the choice of test statistic. 

\begin{lem} Let $g:[0, \, 1] \to \rr$ be a continuous and nondecreasing function. Then 

\beq \label{equation:lemma41}
\int_0^1 x g(x) \d x  \geq \frac{1}{2} \int_0 ^1 g(x) \d x 
\eeq

\noindent with equality if and only if $g$ is constant.  If the inequality in (\ref{equation:lemma41})
is reversed, the same conclusion holds.
\label{lemma:41}
\end{lem}

\begin{rem} \rm We observe that 

$$
\begin{aligned}
\ee[\hat{\beta}_n] &= \frac{12}{n(n^2 -1)} \sum_{i=1}^n \left( i - \frac{n+1}{2}\right) (\lambda_{n,\, i} - \bar{\lambda}_n) \\
&= \frac{12}{n^2-1} \left( \sum_{i=1}^n \frac{i}{n} \cdot \lambda_{n,\,i} - \frac{n +1}{2} \bar{\lambda}_n \right) \\
&= \frac{12n}{n^2 - 1} \left( \int_0 ^1 x g(x) \d x - \frac{1}{2} \int_0^1 g(x) \d x + o(1) \right) \, \, , 
\end{aligned}
$$

\noindent where the last line follows since, for a bounded continuous function $f$, the Riemann sums of the function $f$ over an interval $[a,\, b]$ converge to the Lebesgue integral over $[a,\, b]$. In particular, it follows that, to first order, the mean of the regression coefficient $\hat{\beta}_n$ will be greater than or less than zero if and only if $g$ is strictly monotone increasing or decreasing, respectively. \qed

\end{rem}

We now provide a result furnishing the asymptotic distribution of the test statistic. 

\begin{thm} Let $\{X_{n,\, i}: n \in \nn, \, i \in [n]\}$ be a row-wise strictly stationary triangular array of real-valued random variables such that, for all $n\in \nn$, for some $\delta >0$, $\lVert  X_{n,i} \rVert_{2 + \delta} < \infty$
 and
$$
\sum_{j=1}^{n-1} \alpha_{X_{n,\, \cdot}} (j) ^\frac{\delta}{2+ \delta} < \infty  \, \, .
$$

\noindent Suppose also that there exists $\tau^2 > 0$ such that, as $n \to \infty$,

\beq
\var(X_{n,\,1} ) + 2 \sum_{i=1}^{n-1} \cov(X_{n,\,1}, \, X_{n,\,1+ i} ) \to \tau^2 \, \, .
\label{lim.variance.ols.monotone}
\eeq

\noindent For each $n \in \nn$ and $i \in [n]$, let  

$$
Y_{n, \, i} = X_{n, \, i} + \lambda_{n,\,i} \, \, ,
$$

\noindent where $\{\lambda_{n,\, i}: n \in \nn, \, i \in [n]\} \subset \rr$ is a triangular array of constants. Then, as $n \to \infty$, 

$$
\frac{1}{n^{3/2}} \sum_{i=1}^n \left(i - \frac{n+1}{2} \right) (Y_{n,\, i} - \bar{Y}_n) - \frac{1}{n^{3/2} } \left( \sum_{i=1}^n i \lambda_{n,\, i} - \frac{n(n+1)}{2} \bar{\lambda}_n \right)  \overset{d}\to N(0, \, \tilde{\tau}^2) \, \, ,
$$

\noindent where 

$$
\tilde{\tau}^2 = \frac{1}{12} \tau^2 \, \, .
$$

\label{thm.ols.monotone.clt}
\end{thm}

\begin{rem} \rm Since the $\lambda_{n,\, i}$ are a sequence of constants, it follows that $\var(X_{n,\, 1}) = \var(Y_{n,\, 1})$, and, for all $j\geq 2$, $\cov(X_{n,\, 1}, \, X_{n, \, j}) = \cov(Y_{n,\, 1}, \, Y_{n,\, j})$. \qed

\end{rem}

\subsection{Discussion of the problem setting}

In this section, we make some smoothness assumptions on the additive constants $\lambda_{n,\, i}$; in particular, we assume that the $\lambda_{n,\, i}$ may be written as 

$$
\lambda_{n,\, i} = g\left( \frac{i}{n} \right) 
$$

\noindent for some continuous monotone function $g$, which is accompanied by the implicit assumption that $g:[0,\, 1]\to \rr$ is bounded on $[0,\, 1]$. These assumptions are similar to those made in the setting of \cite{wavk}. The reader may wonder why we have imposed such a restriction on the triangular array of sequences $\{Y_{n,\,i}: n \in \nn,\, i \in [n]\}$, when a more natural choice of problem setting would be to consider the sequence 

\beq
Y_n = X_n + \lambda_n \, \, , 
\label{degenerate.setting}
\eeq

\noindent for $\{X_n: n \in \nn\}$ a sequence of real-valued, mean zero, $\alpha$-mixing random variables, and $\{\lambda_n: n \in \nn\} \subset \rr$ a monotone sequence of constants. 

The reasoning behind this is as follows. In essence, without restriction of the space of additive constants to the form provided, the permutation distribution may degenerate, resulting in the outcome of the test being determined by a single high-leverage observation as opposed to the alternative hypothesis. In order to provide a pathological example of the issues that may occur without this added restriction, we provide the following lemma. 

\begin{lem} Let $n \in \nn$, and let $\{a_i: i \in [n]\}$ and $\{b_i: i \in [n]\}$ be nondecreasing sequences of real numbers. Then, for all $\pi_n \in S_n$, 

$$
\sum_{i=1}^n a_i b_{\pi_n(i)} \leq \sum_{i=1}^n a_i b_i \, \, .
$$

\label{lem.deranged}
\end{lem}

With this lemma in hand, we may now provide an example of the degenerate behavior that may occur in a more general setting. 

\begin{example} (Deterministic, fast-growing sequence) \rm Consider the setting (\ref{degenerate.setting}) in which $X_n \equiv 0$ for all $n \in \nn$, and the constants $\lambda_n$ satisfy, as $n \to \infty$, 

\beq
\frac{n \sum_{i=1}^{n-1} \lambda_i} {\lambda_n} = o(1)  \, \, .
\label{lambda.gets.big}
\eeq

\noindent For instance, we may consider the sequence defined by $\lambda_1 = 1$, and, for $n  \ge 2$, $\lambda_n = \exp(\exp(\lambda_{n-1} ))$. For such a sequence, the sample regression coefficient is equal to 

$$
\hat{\beta}_n = \frac{6}{n^2 } \lambda_n (1 + o(1)) \, \, .
$$

\noindent Let $\Pi_n \sim \text{Unif}(S_n)$. We have that 

$$
\hat{\beta}_{\Pi_n} = \frac{12}{n^3} \left( \Pi_n ^{-1} (n) - \frac{n+1}{2} \right) \lambda_n (1 + o(1)) \, \, .
$$

\noindent In particular, it follows that, for $n$ sufficiently large, the permutation distribution $\hat{R}_n$, based on the test statistic $\hat{\beta}_n$, is approximately equal to the uniform distribution on the set $\{12\lambda_n (i - (n+1)/2)/n^3: i \in [n]\}$, and so has variance $\sigma^2$ approximately equal to

$$
\sigma^2 \approx \frac{12}{n^6} (n^2- 1) \lambda_n^2 = \frac{12}{n^4} (1 + o(1) )\lambda_n^2 \,\, .
$$

\noindent Broadly speaking, the permutation test rejects for large values of $\hat{\beta}_n/\sigma$. We have that

$$
\frac{\hat{\beta}_n}{\sigma} \to \frac{6}{\sqrt{12}} = \sqrt{3} \, \, ,
$$

\noindent a finite constant which does not depend on the sequence $\{\lambda_n\}$ (subject to  (\ref{lambda.gets.big})). 

The degeneracy is more severe still. Since $\{\lambda_i\}$ is nondecreasing, Lemma \ref{lem.deranged} shows that $\hat{\beta}_{\pi_n}$ is maximized at the identity permutation $\pi_n = e_n$; the observed value $\hat{\beta}_n$ is the largest point of its own permutation distribution, and $\sqrt{3}$ is the supremum of the support of the limiting law. The permutation test consequently rejects, but with $p$-value approximately equal to $1/n$, irrespective of the values the $\{\lambda_i\}$ may take. Therefore, at any fixed level $\alpha > 0$, the test rejects on account of the rank of the single high-leverage observation $\lambda_n$ (for sufficiently large $n$), and not due to the trend exhibited by the sequence $\{\lambda_i\}$.

Furthermore, this is not an issue which may be easily resolved by studentization. Since the permutation test must be asymptotically valid under the assumption that $\lambda_n \equiv 0$ for all $n$, examination of the limiting variance
(\ref{lim.variance.ols.monotone}) in Theorem \ref{thm.ols.monotone.clt} reveals that any potential studentization factor $\hat{\tau}_n$ must provide some estimate of the covariances $\cov(X_1, \, X_{1 + i})$, and so should be of the form

$$
\hat{\tau}_n ^2 = \frac{1}{n} \sum_{i=1}^n (Y_{n,\,i} - \bar{Y}_n )^2 + \frac{2}{n} \sum_{i=1}^{b_n} \sum_{j=1}^{n-i} (Y_{n,\, j} - \bar{Y}_n)(Y_{n,\, j + i} - \bar{Y}_n) \, \, .
$$

\noindent The leading term of $\hat{\tau}_n^2$ is a sum of squares, and is therefore invariant under permutation of the $Y_{n, \, i}$. Under (\ref{lambda.gets.big}), we have that 

$$
\frac{1}{n} \sum_{i=1}^n (Y_{n,\,i} - \bar{Y}_n )^2= \frac{\lambda_n^2}{n} (1 + o(1)) \, \, .
$$ 

\noindent In contrast, 

$$
\frac{2}{n} \sum_{i=1}^{b_n} \sum_{j=1}^{n-i} (Y_{n,\, j} - \bar{Y}_n)(Y_{n,\, j + i} - \bar{Y}_n) = O\left(\frac{b_n\lambda_n^2}{n^2}\right) \, \, ,
$$

\noindent and this is dominated by the first sum. Note also that this holds uniformly over $\pi_n \in S_n$. Hence $\hat{\tau}_n = \lambda_n n^{-1/2}(1 + o(1))$. As a result, the ratio under investigation is essentially unaltered by studentization, leading to identical issues to those in the unstudentized setting. \qed

\label{example:41}
\end{example}

In essence, the assumptions on the form of the $\lambda_{n,\, i}$ remove the issues alluded to in the above example, in addition to being similar to the requirements for convergence of a monotone nonparametric regression estimator due to \cite{mukherjee}, in the sense that the restrictions prohibit the existence of extremely high-leverage points in the regression setting under consideration.

Having motivated the problem setting, we may now begin construction of a permutation testing framework which is asymptotically valid under the null hypothesis for weakly dependent sequences, and which is finite-sample exact for i.i.d. triangular arrays $\{Y_{n,\, i}\}$.

\subsection{The limiting permutation distribution}

We begin by considering the limiting permutation distribution $\hat{R}_n$ based on an unstudentized version of the test statistic $\hat{\beta}_n$. To this end, we provide the following result.

\begin{thm} Let $\{X_{n,\, i}: n \in \nn\}$ 
satisfy the conditions of
Theorem \ref{thm.ols.monotone.clt}.
Assume
$
\var(X_{n,\, i}) = \sigma^2 > 0$.
 For each $n \in \nn$ and $i \in [n]$, let  

$$
Y_{n, \, i} = X_{n,\, i} + \lambda_{n,\, i} \, \, ,
$$

\noindent where $\{\lambda_{n,\, i}: n \in \nn, \, i \in [n]\} \subset \rr$ is a triangular array of constants such that, as $n\to\infty$, 

$$
\frac{1}{n} \sum_{i=1}^n (\lambda_{n, \, i} - \bar{\lambda}_n )^2 \to 0 \, \, .
$$

\noindent Let 

$$
T_n = T_n(Y_{n, \,1}, \, \dots,\,Y_{n, \, n}) := \frac{1}{n^\frac{3}{2}} \sum_{i=1}^n \left(i - \frac{n+1}{2} \right)  (Y_{n,\, i} - \bar{Y}_n) \, \, .
$$

\noindent Let $\hat{R}_n$ be the permutation distribution defined in (\ref{equation:Rnt}), based on the test statistic $T_n$ with associated group of transformations $S_n$. Then, as $n\to\infty$, 

$$
\sup_{t \in \rr} \left| \hat{R}_n (t) - \Phi\left( \frac{\sqrt{12}}{\sigma} t \right) \right| \overset{p} \to 0 \, \, ,
$$

\noindent where $\Phi$ is the standard Gaussian distribution. 

\label{ols.monotone.unstud.perm}
\end{thm}

\begin{rem} \rm Note that, in this particular problem setting, we have that, as $n\to \infty$,
$$
\begin{aligned}
\frac{1}{n} \sum_{i=1}^n (\lambda_{n,\, i} -\bar{\lambda}_n )^2 &= \sum_{i=1}^n \frac{1}{n} \lambda_{n,\, i} ^2 - \bar{\lambda}_n^2 \\
&\to \int_0^1 g(x)^2 \d x -\left( \int_0^1 g(x) \d x \right)^2 \, \, .
\end{aligned}
$$

\noindent We have that 

$$
 \int_0^1 g(x)^2 \d x -\left( \int_0^1 g(x) \d x \right)^2  = 0 \iff g(x) = C \text{ for all } x \in [0, \, 1] \, \, ,
$$

\noindent i.e. the permutation distribution only converges under the null hypothesis of no trend.

\end{rem}

We observe that, in general, the variance of the limiting permutation distribution and the variance of the limiting distribution of the test statistic are not equal. We therefore proceed to find an appropriate studentizing factor in order to resolve this issue of mismatched variances, using a similar argument to the one presented in Section 3 of \cite{tirlea}.

\begin{lem} Let $\{X_{n,\, i}: n \in \nn\}$ satisfy the conditions of Theorem \ref{ols.monotone.unstud.perm}.
Let $\{b_n: n \in \nn\} \subset \rr$ be a sequence such that $b_n = o(\sqrt{n})$ and, as $n \to \infty$, $b_n \to \infty$. Let 

$$
\hat{\tau}_n^2 = \frac{1}{n}\sum_{i=1}^n (X_{n,\, i} - \bar{X}_n) ^2 + \frac{2}{n} \sum_{i=1}^{b_n} \sum_{j = 1}^{n- i} (X_{n,\,j}- \bar{X}_n) (X_{n,\, j + i} - \bar{X}_n) \, \, .
$$

\noindent Let $\Pi_n \sim \text{Unif} (S_n)$, independent of the row $\{X_{n,\, i}: i \in [n]\}$. Then the following results hold.

\noindent (i)  As $n \to \infty$, 
$\hat{\tau}_n^2 \overset{p}\to \tau^2 \, \, .$

\noindent (ii)  As $n \to \infty$, $
\hat{\tau}_{\Pi_n}^2 \overset{p}\to \sigma^2 \, \, .$

\label{lem.ols.monotone.stud}
\end{lem}

Finally, by combining Theorem \ref{thm.ols.monotone.clt}, Theorem \ref{ols.monotone.unstud.perm}, and Lemma \ref{lem.ols.monotone.stud} with an application of Slutsky's theorem and Slutsky's theorem for randomization distributions (\cite{chung.perm}, Theorem 5.2), we obtain the following result.

\begin{thm} Assume $\{X_{n,\, i}: n \in \nn\}$  satisfy the assumptions of Theorem \ref{ols.monotone.unstud.perm}.
Let 
$$
\hat{\beta}_n = \frac{1}{\sum_{i=1}^n \left( i - \frac{n+1}{2} \right)^2 } \sum_{i=1}^n \left( i - \frac{n+1}{2} \right) \left( X_{n,\, i}- \bar{X}_n\right) \, \, .
$$

\noindent Then the following results hold.

\noindent (i)  As $n\to \infty$, 

$$
\frac{n^{3/2}\hat{\beta}_n}{ \sqrt{12} \hat{\tau}_n}\overset{d}\to N(0,\, 1)\, \, .
$$

\noindent (ii)  Let $\hat{R}_n$ be the permutation distribution defined in (\ref{equation:Rnt}), based on the test statistic $n^{3/2} \hat{\beta}_n/(\sqrt{12} \hat{\tau}_n)$ with associated group of transformations $S_n$. Then, as $n\to \infty$, 

$$
\sup_{t \in \rr} \left| \hat{R}_n (t) - \Phi(t) \right| \overset{p}\to 0 \, \, ,
$$

\noindent where $\Phi$ is the standard Gaussian c.d.f.

\label{ols.monotone.big.thm}
\end{thm}

\begin{proof} (i) follows immediately from Theorem \ref{thm.ols.monotone.clt}, Lemma \ref{lem.ols.monotone.stud}, and Slutsky's theorem. (ii) follows from Theorem \ref{ols.monotone.unstud.perm}, Lemma \ref{lem.ols.monotone.stud}, and Theorem 5.2 of \cite{chung.perm}. \end{proof}

 As a consequence of Theorem \ref{ols.monotone.big.thm}, the one-sided permutation test of $H_0 ^{(m)}$, based on the test statistic $n^{3/2} \hat{\beta}_n/(\sqrt{12} \hat{\tau}_n)$ is asymptotically valid at the nominal level $\alpha$.

\begin{example}\label{ols.monotone.arma} ($ARMA$ process) \rm Let $\{X_n: \, n \in \nn \} \subset \rr$ satisfy the equation 

\beq
\sum_{i = 0} ^s B_i X_{t - i} = \sum_{k = 0} ^q A_k \epsilon_{t - k} \, \, ,
\label{eq.arma.ex.monotone}
\eeq

\noindent where the $\epsilon_k$ are independent and identically distributed, and $\ee \epsilon_k =0$, i.e. $X$ is an ARMA$(s, \, q)$ process. Let 

\beq
P(z) := \sum_{i=0} ^s B_i z^i \, \, .
\eeq

\noindent If the equation $P(z) = 0$ has no solutions inside the unit circle $\{ z \in \mathbb{C}: \left | z \right| \leq 1\}$, there exists a unique stationary solution to (\ref{eq.arma.ex.monotone}). By \cite{mokkadem}, Theorem 1, if the distribution of the $\epsilon_k$ is absolutely continuous with respect to Lebesgue measure on $\rr$, and if, for some $\delta >0$, 
$\lVert \epsilon_1 \rVert_{4 + 2 \delta} < \infty$,
then the sequence $\{X_n: \, n \in \nn \}$ satisfies the conditions of Theorem \ref{ols.monotone.big.thm}, as long as 

\beq
\tau^2: = \var(X_1) + 2 \sum_{n \geq 1} \cov(X_1, X_{1+ n})
\label{tau2.ols.mon.ar1}
\eeq

\noindent is finite and positive. Therefore, asymptotically, the rejection probability of the permutation test applied to such a sequence will be equal to the nominal level $\alpha$. \qed

\end{example}

Having shown that this permutation test is asymptotically valid, we may now consider the power of this test against local alternatives. Therefore, by an application of Theorem \ref{thm.ols.monotone.clt}, we have the following.

\begin{thm} In the setting of Theorem \ref{ols.monotone.big.thm}, let $\{\lambda_{n,\, i}: n\in \nn, \, i\in [n]\} \subset \rr$ be a triangular array of strictly increasing constants. Let 

$$
Y_{n,\, i} = X_{n,\, i} + \lambda_{n,\, i} \, \,.
$$

\noindent Let $\pp_n$ be the measure induced by the row $\{X_{n,\, i}: i \in [n]\}$, 
and let $\Qq_n$ be the measure induced by the row $\{Y_{n,\, i}: i \in [n]\}$. Suppose that $\{\Qq_n\}$ is contiguous
 to $\{\pp_n\}$. Suppose further that, as $n \to \infty$, 

$$
\frac{1}{n^{3/2}} \left( \sum_{i=1}^n i \lambda_i - \frac{n(n+1)}{2} \bar{\lambda}_n \right) \to \nu \, \, ,
$$

\noindent for some $\nu \in (0, \, \infty)$. Let $\phi_n$ be the test function corresponding to the permutation test of Theorem \ref{ols.monotone.big.thm} conducted at the nominal level $\alpha \in (0, \, 1)$ on the data $\{X_{n,\, i}: i \in [n]\}$. Then, as $n\to \infty$, under $\Qq_n$,

$$
\ee_{\Qq_n} \phi_n \to 1- \Phi\left(z_{1- \alpha} - \frac{\nu\sqrt{12}}{\tau}\right) \,\,,
$$

\noindent where $\tau$ is as in (\ref{lim.variance.ols.monotone}), and $z_{1-\alpha}$ is the $(1- \alpha)$ quantile of the standard Gaussian distribution.

\label{thm.ols.monotone.local} 
\end{thm}

\begin{example}\label{local.ols.monotone} (Local $AR(1)$ alternatives) \rm Let $\rho \in (-1, \, 1)$, and let $\{\epsilon_n: n \in \nn\}$ be an i.i.d. sequence of standard Gaussian random variables. Let $\{X_n: n \in \nn\}$ be the unique stationary sequence satisfying the equation, for $n \geq 1$, 

$$
X_{n +1 } = \rho X_n + \epsilon_n \, \, ,
$$

\noindent i.e. $\{X_n\}$ is an $AR(1)$ process. Clearly, the sequence $\{X_n\}$ has finite moments of every order. In order for $\{X_n\}$ to satisfy the conditions of Theorem \ref{ols.monotone.big.thm}, by Example \ref{ols.monotone.arma} it suffices to show that $\tau^2$, as defined in (\ref{tau2.ols.mon.ar1}), is strictly positive. We have that, for $n \geq 1$, 

$$
\begin{aligned}
\var(X_1) &= \frac{1}{1- \rho^2} \\
\cov(X_1, \, X_{1 + n}) &= \frac{\rho^n}{1 - \rho^2}  \, \, ,
\end{aligned}
$$

\noindent and so 

$$
\begin{aligned}
\tau^2 &= \frac{1}{1 - \rho^2} \left( 1 +2  \sum_{n\geq 1} \rho^n \right) \\
&= \frac{1}{(1- \rho)^2 } \, \, .
\end{aligned}
$$

\noindent Let $h >0$. Consider the triangular array, for $n \in \nn$ and $i \in [n]$, 

$$
\lambda_{n, \, i} = \frac{hi } {n^{3/2} } \, \, .
$$

\noindent Note that, in this setting, the triangular array $\{\lambda_{n,\, i}\}$ corresponds to the sequence of functions $\{g_n: n \in \nn\}$ given by $g_n: [0,\ 1] \to \rr$, 

$$
g_n (x) = \frac{1}{\sqrt{n}} x \, \, .
$$

\noindent Consider the one-sided permutation test of the null hypothesis $H_0^{(m)}$. In the setting of Theorem \ref{thm.ols.monotone.local}, let $X_{n,\, i} = X_i$, and let $Y_{n,\,i} = X_{n,\,i} + \lambda_{n,\,i}$. Let $\{\pp_n: n \in \nn\}$ and $\{\Qq_n: n \in \nn\}$ be as defined in Theorem \ref{thm.ols.monotone.local}. Let $l_n$ be the corresponding log-likelihood ratio. Let $\mu$ be  Lebesgue measure on $\rr$. We have that 

\begin{adjustwidth}{-1cm}{0cm}
$$
\begin{aligned}
l_n :&=  \log\left ( \frac{\d \Qq_n}{\d \pp_n} \right ) \\ 
 &= \log\left( \frac{\d \Qq_n}{\d \mu} ( Y_{n,\,1}, \, \dots,\, Y_{n,\, n})\right) - \log\left( \frac{\d \pp_n} {\d\mu}(X_{n,\,1}, \, \dots,\, X_{n,\, n} )\right)  \\
&= \frac{1}{2(1- \rho^2)} \left[ - \left(X_{n,\,1} -\lambda_{n,\,1} \right)^2 - \sum_{i = 2} ^n \left( X_{n,\, i} - \rho X_{n,\, i-1} - \rho\lambda_{n,\, i-1} - \lambda_{n,\, i}\right)^2 + X_{n,\, 1} ^2 + \sum_{i=2}^n \left( X_{n,\, i} - \rho X_{n,\, i-1}\right)^2  \right] \\
&= \frac{1}{2(1- \rho^2)}  \left[ 2 \lambda_{n,\, 1} X_1 + 2 \sum_{i =2}^n (X_i - \rho X_{i-1})(\lambda_{n,\,i} - \rho\lambda_{n,\,i-1} ) - \lambda_{n,\,1} ^2 - \sum_{i =2} ^n (\lambda_{n,\,i} - \rho \lambda_{n,\, i-1} )^2 \right] \\ 
&= \frac{1}{1 -\rho ^2 } \left[ \left(\lambda_{n,\,1}(1 + \rho^2) - \rho\lambda_{n,\,2} \right) X_1  + \sum_{i=2}^{n-1} \left( (1 + \rho^2) \lambda_{n,\,i} - \rho( \lambda_{n,\, i-1} + \lambda_{n,\, i+1} ) \right) X_i + X_n(\lambda_{n,\, n} - \rho \lambda_{n,\, n-1})\right] - \\
&- \frac{1}{2(1- \rho^2)} \left[ \lambda_{n,\,1} ^2 - \sum_{i =2} ^n (\lambda_{n,\,i} - \rho \lambda_{n,\, i-1} )^2\right]  \\
&=  \frac{1}{(1- \rho^2)} \left[- \frac{1}{2}\left( \lambda_{n,\,1} ^2 - \sum_{i =2} ^n (\lambda_{n,\,i} - \rho \lambda_{n,\, i-1} )^2 \right) + \sum_{i =2} ^n (1- \rho)^2 \lambda_{n,\, i} X_i \right] + o_p (1) \,\,.
\end{aligned}
$$
\end{adjustwidth}

\noindent Letting $\lambda_{n, \, 0} = 0$, we have that 

$$
\begin{aligned}
\sum_{i=1}^n (\lambda_{n,\, i} - \rho \lambda_{n,\, i-1} )^2 &= \sum_{i=1}^n \frac{h^2 i ^2}{n^3} - 2\rho \sum_{i=1}^n \frac{h^2i (i-1) }{n^3} + \rho^2  \sum_{i=1} ^{n-1} \frac{h^2 i^2} {n^3} \\
&= \frac{h^2}{n^3} \left( \frac{n (n+1)(2n + 1)}{6} -2\rho \cdot \frac{n(n^2 - 1)}{3} + \rho^2 \frac{(n-1)n(2n -1)}{6} \right) \\
&= \frac{h^2}{3} (1- \rho)^2 \left( 1 + o(1) \right) \, \, .
\end{aligned} 
$$

\noindent By an analogous argument to the one presented in the proof of Theorem \ref{thm.ols.monotone.clt}, we have that 

$$
\sum_{i=1}^n \lambda_{n,\, i} X_i \overset{d} \to N\left( 0,\, \frac{h^2}{3} \tau^2 \right) \, \, .
$$

\noindent In particular, it follows that, under $\pp_n$, by Slutsky's theorem,

$$
l_n \overset{d} \to N\left( - \frac{h^2}{6}(1- \rho)^2, \, \frac{h^2}{3} (1- \rho)^2 \right) \, \, .
$$

\noindent Therefore, by Le Cam's First Lemma (see \cite{tsh}, Theorem 14.3.1 and Corollary 14.3.1), we have that $\{\Qq_n\}$ and $\{\pp_n\}$ are mutually contiguous. Therefore we may apply Theorem \ref{thm.ols.monotone.local} in this setting. It remains to compute $\nu$ as defined in Theorem \ref{thm.ols.monotone.local}; we have that 

$$
\begin{aligned}
\frac{1}{n^{3/2}} \left( \sum_{i=1}^n i \lambda_{n,\, i} - \frac{n(n+1)}{2} \bar{\lambda}_n \right) &= \frac{h}{n^3} \left( \frac{1}{3} n^3 - \frac{1}{4} n^3 \right) + o(1) \\
&= \frac{h}{12} + o(1) \, \, .
\end{aligned}
$$

\noindent It therefore follows that, in this setting, the local limiting power function of the nominal level $\alpha$ one-sided permutation test is given by 

$$
\ee_{\Qq_n} [\phi_n] \to 1 - \Phi\left( z_{1- \alpha} - \frac{h(1- \rho)}{\sqrt{12}}  \right) \, \, ,
$$

\noindent as $n \to \infty$. \qed

\end{example}

\section{Simulation results} \label{sec.ols.sim}

In this section, we provide Monte Carlo simulations illustrating our results. 
In particular, the results of this paper are asymptotic.  Therefore, we seek to simulate the finite sample rejection probabilities under various null hypothesis scenarios, both in the regression setting and the testing for trend setting,  and compare them with the nominal level $\alpha$.
Subsection \ref{subsec.ols.sim} provides simulation results of the performance of the multivariate least squares permutation test for the unconditional regression coefficient, and Subsection \ref{subsec.monotone.ols.sim} illustrates the performance of the OLS-based test for monotone trend. In both sections, we consider tests conducted at the nominal level $\alpha = 0.05$. The simulation results confirm that the two permutation tests are valid in that, in large samples, the rejection probability under the null hypothesis is approximately equal to $\alpha$.

\subsection{Testing for regression coefficients}\label{subsec.ols.sim}

In this subsection, we provide simulation results for tests of significance in the OLS regression setting. In particular, we provide a comparison of the performance of the two-sided studentized permutation test for significance against the classical two-sided Chi-squared test for significance. As a review, in classical OLS theory, we assume the approximate distribution 

$$
\sqrt{n}(\hat{\beta}_n - \beta) \sim N\left(0, \, \sigma_Y^2 \Sigma_X ^{-1} \right) \, \, .
$$

\noindent Subsequently, the classical Chi-squared test computes the statistic

$$
C_n = C_n((Y_1,\,X_1), \, \dots\, (Y_n, \, X_n)) = \frac{n}{\hat{\sigma}_Y^2} \hat{\beta}_n ^T \hat{\Sigma}_X  \hat{\beta}_n \, \, ,
$$
where $\hat \beta_n$ is defined in (\ref{equation:betanhat}), $\hat \Sigma_X$ is the sample covariance matrix of $X$ defined under (\ref{gam.ols.true.var}) and $\hat \sigma_Y^2$ is the sample variance of $Y$.
Under the null hypothesis $H_0$, $C_n$ is approximately distributed according to the  $\chi_p ^2$ distribution under the classical assumptions, and so the classical  test compares $C_n$ to the $(1-\alpha)$ quantile of the $\chi_p ^2$ distribution in order to reject or not reject the null hypothesis. 

We compare the simulated rejection probabilities of the studentized permutation test with the choice of spherical rejection region $A$ in Remark \ref{rem.rej.ols}; to be precise, we consider the permutation test based on the test statistic 

$$
\tilde{C}_n = n \hat{\beta}_n ^T \hat{\Sigma}_X \hat{\Gamma}_n^{-1} \hat{\Sigma}_X \hat{\beta}_n \, \, ,
$$
where $\hat \Gamma_n$ is defined in (\ref{equation:hatGamma}).  That is, for each permutation of $X$, $\tilde C_n$ is recomputed to build up the permutation distribution (\ref{equation:Rnt}).  If the value of $\tilde C_n$ computed on the original data falls in the upper $\alpha$ tail of the permutation distribution, $H_0$ is rejected.

Note that, as a result of Theorem \ref{ols.perm.big.thm}, under $H_0$, the permutation distribution based on the unstudentized sample regression coefficient $\sqrt{n}\hat{\beta}_n$ converges weakly in probability to the same limiting distribution as $\sqrt{n}\hat{\beta}_n$ under the additional assumption that the sequence $\{X_n: n \in \nn\}$ is i.i.d. and independent of $\{Y_n: n \in \nn\}$. It follows that the unstudentized permutation test will exhibit the same limiting behavior as the classic Chi-squared test of significance, and so we omit simulation of its rejection probabilities in the following analysis.

In this simulation, we consider the following two settings. In Table \ref{tab.mdep.ols.multivar}, we consider processes of the following form. Let $p, \, m \in \nn$. Let $\{\xi_n: n \in \nn\}$ and $\{\xi_{n, \, i} ': n \in \nn, \, i \in [p+1] \}$ be two independent sequences of i.i.d. standard Gaussian random variables. For each $n \in \nn$, let 

$$
T_n = \prod_{j=0}^{m} \xi_{n + j} \, \, .
$$

\noindent For $m = 0$, we set $T_n \equiv 1$. For each $n \in \nn$ and $i \in [p]$, let  

$$
\begin{aligned}
Y_n = T_n \xi'_{n,\, p+1} \\
X_{n,\, i} = T_n \xi'_{n,\, i} \, \, .
\end{aligned}
$$

\noindent Note that, for $m = 0$, the $\{X_n\}$ and $\{Y_n\}$ are i.i.d. Gaussian random variables. We observe that the sequence $\{Z_n = (Y_n, \, X_n): n \in \nn\} \subset \rr^{p+1}$ is $m$-dependent and stationary, and also that the $Z_i$ have finite moments of every order. In particular, it follows that the conditions of Theorem \ref{ols.perm.big.thm} apply, and, since $\cov(Y_1, \, X_1) = 0$, we have that $\beta = 0$. Therefore the asymptotic rejection probability of such a test in this situation should be equal to the nominal level $\alpha$.

\begin{table}[h]
\centering 
\begin{tabular}{cc rrrrrrr} 
\hline\hline 
$m$&$n$&50  &100&    500  &  1,000 &  10,000   \\ [0.5ex]
\hline 
\multirow{2}{*}{0}	&Stud. Perm.&0.043 &0.051&0.067&0.047&0.049\\
				&Classical Chi-Sq. &0.045&0.039&0.046&0.053&0.054\\
				 \hline
\multirow{2}{*}{1}&Stud. Perm.& 0.139&0.158&0.121&0.098&0.070\\
				&Classical Chi-Sq.&0.677&0.716&0.772&0.800&0.833 \\
				 \hline
\multirow{2}{*}{2}&Stud. Perm.&0.222&0.218&0.211&0.176&0.082\\
				&Classical Chi-Sq.&0.802&0.857&0.915&0.934&0.951 \\
				 \hline
\multirow{2}{*}{3}	&Stud. Perm.&0.287&0.286&0.281&0.258&0.113\\
				&Classical Chi-Sq.&0.848&0.890&0.957&0.983&0.985 \\
\hline \hline
\end{tabular}
\caption{Monte Carlo simulation results for null rejection probabilities for tests of $H_0: \beta = 0$, in an $m$-dependent Gaussian product setting with $p = 3$.} \label{tab.mdep.ols.multivar}
\end{table}

We observe that, in the case of $m = 0$, the classical OLS assumptions hold, and it follows that the classical Chi-squared test controls the probability of Type 1 error at the nominal level $\alpha$. However, for $m > 0$, we observe that the classical Chi-squared test has rejection probabilities far exceeding the nominal level $\alpha = 0.05$. The performance of the studentized permutation test is significantly better in comparison although, for larger values of $m$, the simulated rejection probability is still above the nominal level $\alpha$. However, we observe that the simulated rejection probabilities approach the nominal level as the sample size $n$ increases, and it is expected that the performance would improve for larger values of $n$. To illustrate this further, Figure \ref{fig.mdep.ols} plots the empirical distribution of the test statistic for different values of $n$.

\begin{figure}
\includegraphics[width=1.2\textwidth]{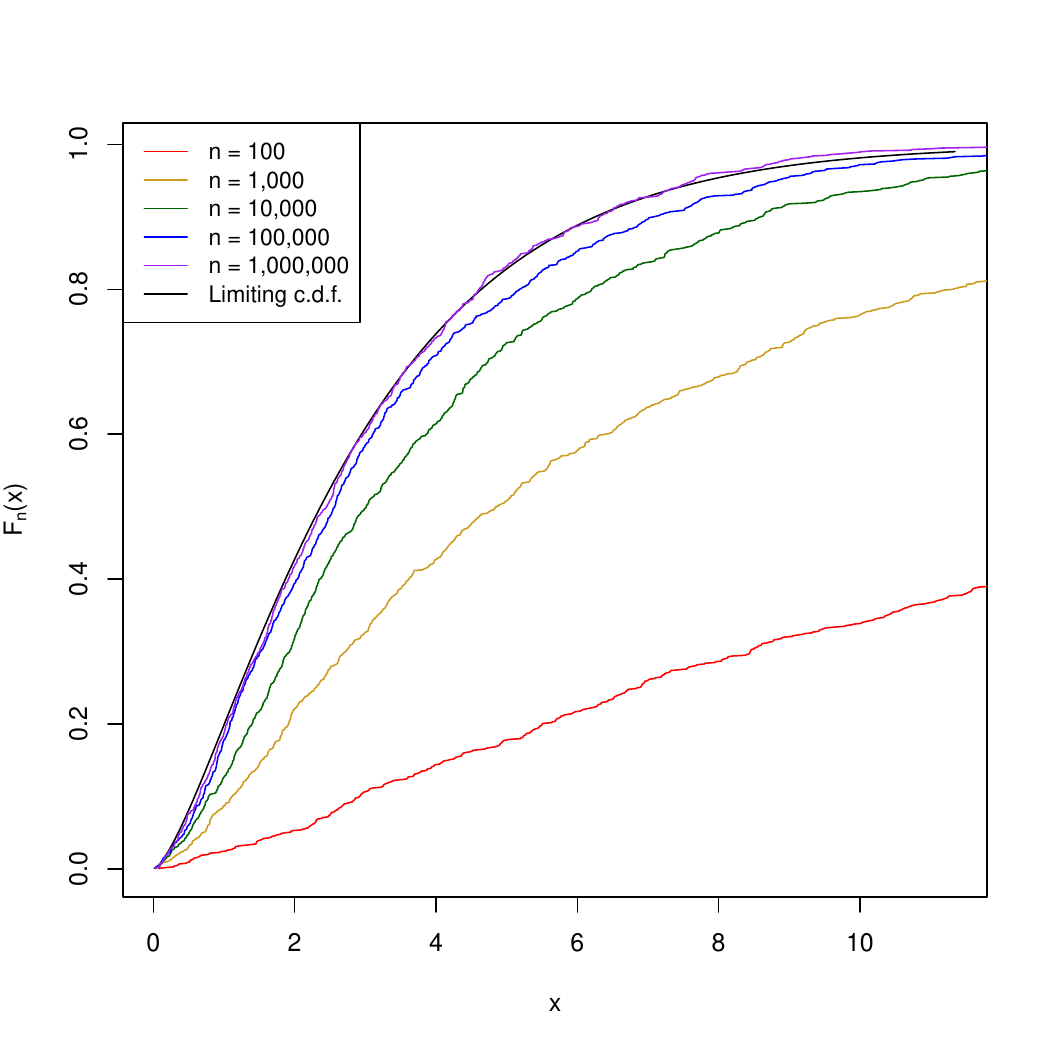}
\caption{Plots of the empirical c.d.f.s of the studentized test statistic $n \hat{\beta}_n^T \hat{\Sigma}_X \hat{\Gamma}_n^{-1} \hat{\Sigma}_X \hat{\beta}_n$ for different values of $n$, in the $m$-dependent setting with $m=3$ and $p=3$.} 
\label{fig.mdep.ols}
\end{figure}


Additionally, Table \ref{tab.ar1.ols.multivar} provides simulations in a $VAR(2)$ setting described in Example \ref{varma.ols}. To be precise, let $\rho \in (-1, \, 1)$, and let 

$$
\begin{aligned}
Q &= \begin{pmatrix} 
\frac{1}{\sqrt{3}}&\frac{1}{\sqrt{2}}&\frac{1}{\sqrt{6}}\\
\frac{1}{\sqrt{3}}&-\frac{1}{\sqrt{2}}&\frac{1}{\sqrt{6}}\\
\frac{1}{\sqrt{3}}&0&-\frac{2}{\sqrt{6}}
\end{pmatrix}\\
D &= \begin{pmatrix}
\rho & 0 & 0\\
0 & \rho & 0 \\
0 & 0 & -\rho \end{pmatrix}\\
R & = Q D Q^T
\end{aligned}
$$

\noindent Let $\{\epsilon_n: n \in \nn\} \subset \rr^3$ be an i.i.d. sequence of standard trivariate Gaussian random variables, and let $\{X_n: n \in \nn\} \subset \rr^3 $ be the stationary process satisfying the $VAR(2)$ equation, for all $n \geq 3$, 

$$
X_n = R X_{n-2} + \epsilon_n \, \, .
$$

\noindent The sequence $\{X_n\}$ is indeed stationary since, by construction, the matrix $Q$ is orthogonal, and so $D$ is exactly the matrix of the eigenvalues of $R$. Having appropriately defined $\{X_n\}$, let 

$$
Y_n = X_{n + 1,\, 1}\, \, .
$$

\noindent By Example \ref{varma.ols}, the sequence $\{ X_n\}$ satisfies the conditions of Theorem \ref{ols.perm.big.thm}. Since the $\alpha$-mixing coefficients of the sequence $\{Z_n = (Y_n,\, X_n): n \in \nn\}$ satisfy, for all $n \in \nn$, $\alpha_Z (n) = \alpha_X (n -1)$, and the $Z_i$ form a stationary sequence with finite moments of every order, it follows that we may apply the result of Theorem \ref{ols.perm.big.thm} in this setting. Also, since $X_{2n -1}$ and $X_{2n}$ are independent for all $n \in \nn$ by construction, it follows that, in this case, the unconditional regression coefficient $\beta = 0$, and so the limiting rejection probability of the studentized permutation test converges to the nominal level $\alpha$.

\begin{table}[h]
\centering 
\begin{tabular}{cc rrrrrrr} 
\hline\hline 
$\rho$&$n$&50  &100&    500  &  1,000 &  10,000   \\ [0.5ex]
\hline 
\multirow{2}{*}{-0.8}	&Stud. Perm.&0.037 &0.079&0.044&0.072&0.051\\
				&Classical Chi-Sq. &0.206&0.214&0.228&0.255&0.221\\
				 \hline
\multirow{2}{*}{-0.5}&Stud. Perm.& 0.053&0.032&0.044&0.072&0.036\\
				&Classical Chi-Sq.&0.051&0.062&0.053&0.053&0.076 \\
				 \hline
\multirow{2}{*}{0.5}&Stud. Perm.&0.069&0.062&0.045&0.054&0.046\\
				&Classical Chi-Sq.&0.136&0.147&0.180&0.175&0.162 \\
				 \hline
\multirow{2}{*}{0.8}	&Stud. Perm.&0.072&0.086&0.059&0.148&0.049\\
				&Classical Chi-Sq.&0.379&0.443&0.477&0.503&0.491 \\
\hline \hline
\end{tabular}
\caption{Monte Carlo simulation results for null rejection probabilities for tests of $H_0: \beta = 0$, in a $VAR(2)$ setting, with $p = 3$.} \label{tab.ar1.ols.multivar}
\end{table}

For each situation, 1,000 simulations were performed. Within each simulation, the permutation test was calculated by randomly sampling 2,000 permutations. 

The rejection probabilities of the studentized permutation test in Table \ref{tab.mdep.ols.multivar} remain above the nominal level for $m>0$ even at the largest sample sizes considered; we now discuss the reason for this and its implications for practical application.

The rows of Table \ref{tab.mdep.ols.multivar} vary two features of the data-generating process. Concretely, for $m \geq 1$, the variable $Y_n$ is a product of $m+2$ independent standard Gaussian random variables, and therefore
has kurtosis equal to $3^{m+2}$, namely $27$, $81$ and $243$ for $m = 1, \, 2$ and $3$ respectively; for $m = 0$, the convention $T_n \equiv 1$ renders the observations i.i.d. Gaussian (and the kurtosis equal to $3$). Descending the table thus increases the order of dependence and the weight of the marginal tails of the distribution simultaneously.

It is the latter which the studentization must accommodate. Since $\beta = 0$, and since the $X_n$ and $Y_n$ have mean zero, the summands defining $\Gamma$ in (\ref{gam.ols.true.var}) reduce to $X_n Y_n$, with components

$$
W_{n, \, k} = X_{n, \, k} Y_n = T_n^2 \xi_{n, \, k}' \xi_{n, \, p+1}' \, \, ,
$$

\noindent where 

$$
\begin{aligned}
\ee [W_{n, \, k}^2 ] &= 3^{m+1} \, \, , \\
\ee [ W_{n, \, k}^4] &= 9 \cdot 105^{m+1} \, \, .
\end{aligned}
$$

\noindent Both $\hat \gamma_n$ and $\hat \tau_n$ are averages of products of pairs of these variables; the effect of this is that the sampling variance of $\hat \Gamma_n$ is controlled by the fourth moment $\ee [ W_{n, \, k}^4]$, which takes the values $9.9 \times 10^4$, $1.0\times 10^7$ and $1.1 \times 10^9$ for $m = 1, \, 2, \,
3$, compared to $9$ in the i.i.d. Gaussian case $m=0$.

The quantity determining how well the studentizer is estimated therefore grows by eight orders of magnitude across four rows of the table, and it is to be expected that substantially larger sample sizes are required for the sampling error in $\hat{\Gamma}_n$ to become negligible at larger values of $m$.

We may ascribe the difficulty to the tails as opposed to the dependence by direct comparison with Table \ref{tab.ar1.ols.multivar}. The $VAR(2)$ process considered there is not $m$-dependent for any $m$: its dependence decays geometrically but persists at every lag, and at $\rho = \pm 0.8$ it is strongly persistent. This is evidenced by the rejection probabilities of the classical
Chi-squared test, which range from $0.206$ to $0.503$. Its marginals, however, are Gaussian, and the studentized permutation test is close to the nominal level at sample sizes as small as $n = 50$, with no systematic deterioration as $|\rho|$ increases.

Accordingly, we recommend application of the studentized permutation test for stationary, weakly dependent data whose marginal distributions have moderate tails, in which setting the simulations above indicate that the nominal level is attained at sample sizes routinely encountered in practice. When the data are in addition very heavy-tailed, the nominal level is approached more slowly, and the relevant diagnostic is not the strength of the serial dependence but the tail behavior of the products $X_i \hat \epsilon_{i, \, n}$. To assess it, the sample kurtosis of each of the sequences $\{X_{i, \, k} \hat{\epsilon}_{i, \, n}: i \in [n]\}$,\,  $k \in [p]$, may be
computed directly once the regression has been fitted. A large value indicates that $\hat{\Gamma}_n$ is poorly determined at the sample size in question, and that the rejection probability may exceed the nominal level. Even in this regime, however, the studentized test is greatly preferable to the classical alternative: the rejection probabilities of the classical Chi-squared test in Table \ref{tab.mdep.ols.multivar} rise to $0.833$, $0.951$ and $0.985$ for $m = 1, \, 2, \, 3$ at $n = 10{,}000$, \textit{increasing} with $n$ rather than diminishing. By comparison, the studentized permutation test's rejection probabilities approach the nominal level, albeit more slowly than in Table \ref{tab.ar1.ols.multivar}.

We provide a brief discussion of several computation choices to be made when performing the aforementioned testing procedure. By the results of Lemma \ref{lem.ols.studentize}, for large values of $n$, the estimates $\hat{\Gamma}_n$ and $\hat{\Sigma}_X$ will be strictly positive definite with high probability. However, for smaller values of $n$, it may be the case that the numerical minimal eigenvalue of $\hat{\Gamma}_n$ may be nonpositive, or that, at least numerically, $\hat{\Sigma}_X$ may be non-positive definite, either when computing the test statistic or the permutation distribution. 

The solution to this issue implemented in the above simulations is as follows. Since $\hat{\Sigma}_X$ is a Gram matrix and is positive semi-definite by construction, we invert it through a symmetric eigendecomposition with its eigenvalues truncated below at some sufficiently small numerical constant $\epsilon >0$. In contrast, the estimator $\hat{\Gamma}_n$ need not be symmetric for $p >1$, nor positive semi-definite; we invert it exactly as stated in (\ref{equation:hatGamma}), through a singular value decomposition with its singular values truncated below at $\epsilon$. These choices are implemented in \texttt{permixOLS}.

We observe that for $\lambda_{\text{min}} (A)$ denoting the minimal eigenvalue of the matrix $A$, choosing $\epsilon < \min\{ \lambda_{\text{min}} (\Gamma), \, \lambda_{\text{min}} (\Sigma_X)\}$ will still allow the result of Lemma \ref{lem.ols.studentize} to hold, i.e. any inference based on this choice of truncated studentization is still asymptotically valid. In practice, however, the suitability of a choice of $\epsilon$ for a particular numerical application is determined by the distribution of the sequence $\{Z_n = (Y_n,\, X_n): n \in \nn\}$ under analysis. For the above simulation, a constant value of $\epsilon = 10^{-4}$ was used. 

A further choice is that of the truncation sequence $\{b_n,\, n \in \nn\}$ used in the definition of $\hat{\Gamma}_n $. Any sequence $  \{b_n\}$ such that, as $n \to \infty$, $b_n \to \infty$ and $b_n = o\left(\sqrt{n} \right)$ is theoretically justified by Theorem \ref{ols.perm.big.thm}, although, in a specific setting, some choices of $\{b_n\}$ will lead to more numerical stability than others. In the above simulation, $\{b_n\}$ was taken to be $[n^{1/3}] + 1$, where $[x]$ denotes the integer part of $x$.

\subsection{Testing for monotone trend using least squares}\label{subsec.monotone.ols.sim}

We now turn our attention to simulations involving the least squares test of the null hypothesis $H_0^{(m)}$ of no monotone trend, against the alternative 
$$
K: g \text{ is strictly increasing.}
$$

\noindent Under the assumption that the sequence $\{X_i: i \in [n]\}$ is i.i.d., the randomization hypothesis holds, and so the unstudentized least squares permutation test will be exact in finite samples and asymptotically valid. However, in general, the randomization hypothesis does not hold under $H_0 ^{(m)}$, and so the unstudentized least squares permutation test will not be exact or even asymptotically valid. However, by the result of Theorem \ref{ols.monotone.big.thm}, the studentized least squares permutation test will be asymptotically valid at the nominal level $\alpha$. 

In order to illustrate this behavior, we provide a comparison of the simulated rejection probabilities of the studentized and unstudentized one-sided least squares permutation tests for lack of monotone trend in two different settings. For both of the following situations, 1,000 simulations were performed. Within each simulation, the permutation test was calculated by randomly sampling 2,000 permutations. 

In Table \ref{tab.mdep.ols.monotone}, we consider processes of the following form. Let $m \in \nn$. Let $\{Z_n: n \in \nn\}$ be an i.i.d. sequence of standard Gaussian random variables. For each $n \in \nn$, let

$$
X_n = \prod_{j =0}^{m} Z_{n + j}
$$

\noindent We observe that the sequence $\{X_n: n \in \nn\}$ is stationary and $m$-dependent, and has finite moments of every order. In particular, the triangular array $X_{n,\, i} = X_i$ satisfies the conditions of Theorem \ref{ols.monotone.big.thm}, and so the one-sided studentized permutation test will be asymptotically valid at the nominal level $\alpha$. 

\begin{table}[h]
\centering 
\begin{tabular}{cc rrrrrrr} 
\hline\hline 
$m$&$n$&20  &50&    100  &  500 &  1000   \\ [0.5ex]
\hline 
\multirow{2}{*}{0}	&Stud. Perm.&0.054 &0.046&0.047&0.039&0.045\\
				&Unstud. Perm.&0.047&0.049&0.051&0.045&0.046\\
				 \hline
\multirow{2}{*}{1}&Stud. Perm.& 0.046&0.051&0.059&0.047&0.051\\
				&Unstud. Perm.&0.056&0.044&0.052&0.046&0.061 \\
				 \hline
\multirow{2}{*}{2}&Stud. Perm.&0.031&0.048&0.061&0.049&0.047\\
				&Unstud. Perm.&0.057&0.055&0.066&0.056&0.045 \\
				 \hline
\multirow{2}{*}{3}	&Stud. Perm.&0.041&0.034&0.058&0.047&0.042\\
				&Unstud. Perm.&0.064&0.048&0.055&0.048&0.043 \\
\hline \hline
\end{tabular}
\caption{Monte Carlo simulation results for null rejection probabilities for tests of $H_0^{(m)}$, in an $m$-dependent Gaussian product setting.} \label{tab.mdep.ols.monotone}
\end{table}

While the simulation results of Table \ref{tab.mdep.ols.monotone} support the theoretical asymptotic validity of the studentized permutation test, we also observe that the unstudentized permutation test exhibits similar Type 1 error control. This may be explained as follows: for any choice of $m \in \nn$, for all $i,\, j\in \nn$ such that $i \neq j$, we have that 

$$
\begin{aligned}
\var(X_i) &= 1\\
\cov(X_i,\, X_j) &= 0 \, \, .
\end{aligned}
$$

\noindent It follows that, in this particular setting, the limiting variances of the test statistic and the limiting permutation distribution match, and so the unstudentized permutation test is also asymptotically valid.

In Table \ref{tab.ar.ols.monotone}, we consider a more traditional $AR(1)$ setting: namely, for $\rho\in (-1, \, 1)$ and $\{\epsilon_n: n \in \nn\}$ a sequence of i.i.d. standard Gaussian random variables, let $\{X_n: n \in \nn\}$ be the unique stationary sequence satisfying the equation, for all $n \in \nn$,

$$
X_{n + 1} = \rho X_n + \epsilon_{n+1} \, \, .
$$

\noindent By Example \ref{ols.monotone.arma}, the sequence $\{X_n\}$ satisfies the conditions of Theorem \ref{ols.monotone.big.thm}, and so the rejection probability of the studentized permutation test converges to the nominal level $\alpha$ as $n \to \infty$. 

\begin{table}[h]
\centering 
\begin{tabular}{cc rrrrrrr} 
\hline\hline 
$\rho$&$n$&20  &50&    100  &  500 &  1000   \\ [0.5ex]
\hline 
\multirow{2}{*}{-0.6}	&Stud. Perm.&0.191 &0.064&0.140&0.056&0.041\\
				&Unstud. Perm.&0.006&0.001&0.000&0.001&0.000\\
				 \hline
\multirow{2}{*}{-0.2}&Stud. Perm.& 0.064&0.065&0.057&0.049&0.050\\
				&Unstud. Perm.&0.027&0.017&0.027&0.020&0.023 \\
				 \hline
\multirow{2}{*}{0.2}&Stud. Perm.&0.030&0.048&0.041&0.045&0.044\\
				&Unstud. Perm.&0.089&0.100&0.085&0.088&0.089 \\
				 \hline
\multirow{2}{*}{0.6}	&Stud. Perm.&0.011&0.025&0.041&0.049&0.043\\
				&Unstud. Perm.&0.181&0.215&0.214&0.200&0.180 \\
\hline \hline
\end{tabular}
\caption{Monte Carlo simulation results for null rejection probabilities for tests of $H_{0}^{(m)}$, in an $AR(1)$ setting.} \label{tab.ar.ols.monotone}
\end{table}

We observe that, while the studentized permutation test exhibits Type 1 error control at the nominal level $\alpha$ for all choices of $\rho$, this is not the case for the unstudentized permutation test. Indeed, for $\rho >0$, the rejection probability of the unstudentized permutation test is significantly higher than the nominal level $\alpha$, i.e. we do not have Type 1 error control. While, for $\rho < 0$, the unstudentized permutation test does control Type 1 error, the rejection probability is observed to be significantly below the nominal level $\alpha$. This can lead to issues of Type 2 error against local alternatives of the form described in Example \ref{local.ols.monotone}, for which this test will have power less than $\alpha$.

As in Subsection \ref{subsec.ols.sim}, there are several computational choices to be made in practice. In particular, on account of the same numerical issues involving the studentization factor $\hat{\tau}_n^2$, we truncate the variance estimate at the lower bound $\epsilon = 10^{-6}$. Similarly, we must choose a truncation sequence $\{b_n: n \in \nn\}$ to be used in the definition of $\hat{\tau}_n^2$. In the above simulations, the choice $b_n = [n^{1/3}] + 1$ was used. 

\section{Conclusions}

When the fundamental assumption of exchangeability does not necessarily hold, permutation tests are invalid unless strict conditions on underlying parameters of the problem are satisfied. For instance, the permutation test of $H_0: \beta = 0$ based on the standard OLS sample regression coefficient is asymptotically valid only when $\Sigma_X^{-1} \Gamma \Sigma_X^{-1}$, as defined in (\ref{gam.ols.true.var}), is equal to $\sigma^2 _Y \Sigma_X ^{-1}$. Hence rejecting the null must be interpreted correctly, since rejection of the null with this permutation test does not necessarily imply that the unconditional regression coefficient $\beta$ is different from zero. 

Similar issues occur in using least squares-based tests for lack of monotone trend. Unless the assumption of exchangeability holds, a permutation test of $H_0^ {(m)}$ based on the sample regression coefficient $\hat{\beta}_n$ is only asymptotically valid if $\tau^2$, as defined in (\ref{lim.variance.ols.monotone}), is equal to the limiting variance of the random variables in the corresponding triangular array, i.e. if $\tau^2 = \lim_{n\to \infty} \var(X_{n,\,1})$. Otherwise, a rejection of the null hypothesis $H_0 ^{(m)}$ may not necessarily imply that the function $g$ is strictly monotone, in the sense that it may be the case that $g$ is constant and the permutation test will reject the null hypothesis.

We provide a testing procedure that allows one to obtain asymptotic rejection probability $\alpha$ in a permutation test setting. A significant advantage of this test is that it achieves finite-sample exactness under the assumption of exchangeability, which is only the case in classical OLS tests of significance under the assumption of i.i.d. and Gaussianity, as well as achieving asymptotic level $\alpha$ in a much wider range of settings than classical tests of significance in OLS settings. Similarly, we provide a testing procedure for the null hypothesis $H_0 ^{(m)}$ which is asymptotically valid in a wide range of weakly dependent settings.

Correct implementation of a permutation test is crucial if one is interested in confirmatory inference via hypothesis testing; indeed, proper error control of Type 1, 2 and 3 errors can be obtained for tests of significance in linear regression or in least-squares-based tests for monotone trend by basing inference on test statistics which are asymptotically pivotal. A framework has been provided for tests in both of these problem settings. 

\section{Acknowledgements}

Both authors were supported by NSF Grant MMS-1949845.

\section*{Declaration of generative AI and AI-assisted technologies in the
manuscript preparation process}

During the preparation of this work the authors used Claude (Anthropic) to
assist in developing the accompanying \texttt{R} package \texttt{permixOLS}, and
to check the manuscript's statements against that implementation. All mathematical content, proofs and simulation designs are the
authors' own. The authors reviewed and edited all resulting text and code, and
take full responsibility for the content of the publication.

\bibliographystyle{elsarticle-num}  
\bibliography{perm_test_thesis}       

\end{document}